\documentclass{amsproc}
\usepackage{amssymb}
\usepackage{graphicx}
\usepackage{amscd}
\usepackage{amsmath}
\theoremstyle{plain}

\newtheorem{definition}{Definition}

\newtheorem{lemma}{Lemma}

\newtheorem{proposition}{Proposition}
\newtheorem{remark}{Remark}

\newtheorem{theorem}{Theorem}

\numberwithin{equation}{section}

\begin{document}

\title{Dynamical inverse problem for Jacobi matrices.}
\author{ A. S. Mikhaylov}
\address{St. Petersburg   Department   of   V.A. Steklov    Institute   of   Mathematics
of   the   Russian   Academy   of   Sciences, 7, Fontanka, 191023
St. Petersburg, Russia and Saint Petersburg State University,
St.Petersburg State University, 7/9 Universitetskaya nab., St.
Petersburg, 199034 Russia.} \email{a.mikhaylov@spbu.ru}

\author{ V. S. Mikhaylov}
\address{St.Petersburg   Department   of   V.A.Steklov    Institute   of   Mathematics
of   the   Russian   Academy   of   Sciences, 7, Fontanka, 191023
St. Petersburg, Russia and Saint Petersburg State University,
St.Petersburg State University, 7/9 Universitetskaya nab., St.
Petersburg, 199034 Russia.} \email{v.mikhaylov@spbu.ru}

\date{July, 2016}

\maketitle

\noindent {\bf Abstract.} We consider the inverse dynamical
problem for the dynamical system with discrete time associated
with the semi-infinite Jacobi matrix. We solve the inverse problem
for such a system and answer a question on the characterization of
the inverse data. As a by-product we give a necessary and
sufficient condition for the measure on the real line line to be
the spectral measure of semi-infinite discrete Schrodinger
operator.

\section{Introduction.}

Given a positive sequence $\{a_0, a_1,\ldots\}$ and real $\{b_1,
b_2,\ldots  \}$ we consider the operator $H$ corresponding to
semi-infinite Jacobi matrix, defined on $l^2\ni
\phi=(\phi_0,\phi_1,\ldots)$, given by
\begin{align*}
(H\phi)_n&=a_{n}\phi_{n+1}+a_{n-1}\phi_{n-1}+b_n\phi_n,\quad
n\geqslant 1,\\
(H\phi)_0&=a_0\phi_1,\quad n=0.
\end{align*}
adding the Dirichlet boundary condition $\phi_0=0$ give rise to
the spectral measure $d\rho$ (in the case when $H$ is in the limit
circle case at infinity, see \cite{A,S}, this measure in non
unique and is paramertized by a point on a unit circle). When all
$a_n=1$, $n=\{0,1,\ldots\}$ the operator $H$ is called the
discrete Schr\"odinger operator. In \cite{SG} the authors set up a
question of the characterization of the spectral measure for the
discrete Schrodinger operator. The goal of the paper is to answer
this question. Below we formulate the main result: let
$\mathrm{T}_k(2\lambda)$ be the Chebyshev polynomials of the
second kind: i.e. they satisfy
\begin{equation*}
\left\{
\begin{array}l
\mathrm{T}_{t+1}+\mathrm{T}_{t-1}-\lambda \mathrm{T}_{t}=0,\\
\mathrm{T}_{0}=0,\,\, \mathrm{T}_1=1.
\end{array}
\right.
\end{equation*}
\begin{theorem}
The measure $d\rho$ is a spectral measure of discrete
Schr\"odinger operator if and only if for every $T\geqslant 1$ the
matrix $C^T$ with the entries
\begin{equation*}
C^T_{i,j}=\int_{-\infty}^\infty
T_{T-i+1}(\lambda)T_{T-j+1}(\lambda)\,d\rho(\lambda),\quad
i,j=1,\ldots,T.
\end{equation*}
is positive definite and $\det{C^T}=1.$
\end{theorem}

We use the dynamical approach: we consider the dynamical system
with discrete time associated with the Jacobi marix, which is a
natural analog of dynamical systems governed by the wave equation
on a semi-axis:
\begin{equation}
\label{Jacobi_dyn} \left\{
\begin{array}l
u_{n,t+1}+u_{n,t-1}-a_{n}u_{n+1,t}-a_{n-1}u_{n-1,t}-b_nu_{n,t}=0,\quad n,t\in \mathbb{N},\\
u_{n,-1}=u_{n,0}=0,\quad n\in \mathbb{N}, \\
u_{0,t}=f_t,\quad t\in \mathbb{N}\cup\{0\}.
\end{array}\right.
\end{equation}
By analogy with continuous problems \cite{B07}, we treat the real
sequence $f=(f_0,f_1,\ldots)$ as a \emph{boundary control}. The
solution to (\ref{Jacobi_dyn}) we denote by $u^f_{n,t}$. Having
fixed $\tau\in \mathbb{N}$, with (\ref{Jacobi_dyn}) we associate
the \emph{response operators}, which maps the control
$f=(f_0,\ldots f_{\tau-1})$ to $u^f_{1,t}$:
\begin{equation*}
\left(R^\tau f\right)_t:=u^f_{1,t},\quad t=1,\ldots, \tau.
\end{equation*}
The inverse problem we will be dealing with is to recover from
$R^{\tau}$ the sequences $\{b_1,b_2,\ldots,b_n\}$,
$\{a_0,a_1,\ldots,a_n\}$ for some $n$. This problems is a natural
discrete analog of the inverse problem for the wave equation where
the inverse data is the dynamical Dirichlet-to-Neumann map, see
\cite{B07}.

To treat the inverse dynamical problem we will use the Boundary
Control method \cite{B07} which was initially developed to treat
multidimensional dynamical inverse problems, but since then was
applied to multy- and one- dimensional inverse dynamical, spectral
and scattering problems, problems of signal processing and
identification problems.

In the second section we study the forward problem: for
(\ref{Jacobi_dyn}) we prove the analog of d'Alembert integral
representation formula, we also introduce and prove the
representation formulaes for the main operators of the BC method:
response operator, control and connecting operators. In the third
section we derive the equations for the inverse problem and give a
characterization of the dynamical inverse data for the case of
Jacobi matrix and for the case of discrete Schr\"odinger operator.
In the last section we derive the spectral representation
formulaes for response and connecting operators and use the
results obtained to prove Theorem 1.

\section{Forward problem, operators of the Boundary Control method.}

We fix some positive integer $T$. By $\mathcal{F}^T$ we denote the
\emph{outer space} of the system (\ref{Jacobi_dyn}), the space of
controls: $\mathcal{F}^T:=\mathbb{R}^T$, $f\in \mathcal{F}^T$,
$f=(f_0,\ldots,f_{T-1})$.

First, we derive the representation formulas for the solution to
(\ref{Jacobi_dyn}) which could be considered as analogs of known
formulas for the wave equation \cite{AM}.
\begin{lemma}
The solution to (\ref{Jacobi_dyn}) admits the representation
\begin{equation}
\label{Jac_sol_rep} u^f_{n,t}=\prod_{k=0}^{n-1}
a_kf_{t-n}+\sum_{s=n}^{t-1}w_{n,s}f_{t-s-1},\quad n,t\in
\mathbb{N}.
\end{equation}
where $w_{n,s}$ satisfies the Goursat problem
\begin{equation}
\label{Goursat} \left\{
\begin{array}l
w_{n,s+1}+w_{n,s-1}-a_nw_{n+1,s}-a_{n-1}w_{n-1,s}-b_nw_{n,s}=-\delta_{s,n}(1-a_n^2)\prod_{k=0}^{n-1}a_k,\,n,s\in \mathbb{N}, \,\,s>n,\\
w_{n,n}-b_n\prod_{k=0}^{n-1}a_k-a_{n-1}w_{n-1,n-1}=0,\quad n\in \mathbb{N},\\
w_{0,t}=0,\quad t\in \mathbb{N}_0.
\end{array}
\right.
\end{equation}
\end{lemma}
\begin{proof}
We assume that $u^f_{n,t}$ has a form (\ref{Jac_sol_rep}) with
unknown $w_{n,s}$ and plug it to equation in (\ref{Jacobi_dyn}):
\begin{eqnarray*}
0=\prod_{k=0}^{n-1}a_kf_{t+1-n}+\prod_{k=0}^{n-1}a_kf_{t-1-n}-a_n\prod_{k=0}^{n}a_kf_{t-n-1}-a_{n-1}\prod_{k=0}^{n-2}a_kf_{t-n+1}\\
-b_n\prod_{k=0}^{n-1}a_kf_{t-n}+\sum_{s=n}^{t}w_{n,s}f_{t-s}+\sum_{s=n}^{t-2}w_{n,s}f_{t-s-2}\\-a_n\sum_{s=n+1}^{t-1}w_{n+1,s}f_{t-s-1}
-a_{n-1}\sum_{s=n-1}^{t-1}w_{n-1,s}f_{t-s-1}-\sum_{s=n}^{t-1}b_nw_{n,s}f_{t-s-1}.
\end{eqnarray*}
Evaluating and changing the order of summation we get
\begin{eqnarray*}
0=\left(1-a_n^2\right)\prod_{k=0}^{n-1}a_kf_{t-n-1}-b_n\prod_{k=0}^{n-1}a_kf_{t-n}\\
-\sum_{s=n}^{t-1}f_{t-s-1}\left(b_nw_{n,s}+a_nw_{n+1,s}+a_{n-1}w_{n-1,s}\right)
+a_nw_{n+1,n}f_{t-n-1}\\
-a_{n-1}w_{n-1,n-1}f_{t-n}
+\sum_{s=n-1}^{t-1}w_{n,s+1}f_{t-s-1}+\sum_{s=n+1}^{t-1}w_{n,s-1}f_{t-s-1}\\
=\sum_{s=n}^{t-1}f_{t-s-1}\left(w_{n,s+1}+w_{n,s-1}-a_nw_{n+1,s}-a_{n-1}w_{n-1,s}-b_nw_{n,s}\right)\\
-b_n\prod_{k=0}^{n-1}a_kf_{t-n}+\left(1-a_n^2\right)\prod_{k=0}^{n-1}a_kf_{t-n-1}+a_nw_{n+1,n}f_{t-n-1}\\
-a_{n-1}w_{n-1,n-1}f_{t-n}+w_{n,n}f_{t-n}-w_{n,n-1}f_{t-n-1}.
\end{eqnarray*}
Finally we arrive at
\begin{eqnarray*}
\sum_{s=n}^{t-1}f_{t-s-1}\left(w_{n,s+1}+w_{n,s-1}-a_nw_{n+1,s}-a_{n-1}w_{n-1,s}-b_nw_{n,s}+\left(1-a_n^2\right)\prod_{k=0}^{n-1}a_k\delta_{sn}\right)\\
+f_{t-n}\left(w_{n,n}-a_{n-1}w_{n-1,n-1}-b_n\prod_{k=0}^{n-1}a_k\right)=0.
\end{eqnarray*}
Counting that $w_{n,s}=0$ when $n>s$ and arbitrariness of $f\in
\mathcal{F}^T$, we arrive at (\ref{Goursat}).
\end{proof}

\begin{definition}
For $a,b\in l^\infty$ we define the convolution $c=a*b\in
l^\infty$ by the formula
\begin{equation*}
c_t=\sum_{s=0}^{t}a_sb_{t-s},\quad t\in \mathbb{N}\cup \{0\}.
\end{equation*}
\end{definition}

As an inverse data for (\ref{Jacobi_dyn}) we use the analog of the
dynamical response operator (dynamical Dirichlet-to-Neumann map)
\cite{B07}.
\begin{definition}
For (\ref{Jacobi_dyn}) the \emph{response operator}
$R^T:\mathcal{F}^T\mapsto \mathbb{R}^T$ is defined by the rule
\begin{equation*}
\left(R^Tf\right)_t=u^f_{1,t}, \quad t=1,\ldots,T.
\end{equation*}
\end{definition}
Introduce the notation: the \emph{response vector} is the
convolution kernel of the response operator,
$r=(r_0,r_1,\ldots,r_{T-1})=(a_0,w_{1,1},w_{1,2},\ldots
w_{1,T-1})$. Then in accordance with (\ref{Jac_sol_rep})
\begin{eqnarray}
\label{R_def}
\left(R^Tf\right)_t=u^f_{1,t}=a_0f_{t-1}+\sum_{s=1}^{t-1}
w_{1,s}f_{t-1-s}
\quad t=1,\ldots,T.\\
\notag \left(R^Tf\right)=r*f_{\cdot-1}.
\end{eqnarray}
 If we take special control
 $f=\delta=(1,0,0,\ldots)$, then the kernel of response operator becomes
\begin{equation}
    \label{con11} \left(R^T\delta\right)_t=u^\delta_{1,t}= r_{t-1}.
\end{equation}

%For the system (\ref{Jacobi_dyn_int}) we introduce the response
%operator by
%\begin{definition}
%For the system in (\ref{Jacobi_dyn_int}) the \emph{response
%operator} $R^T_i:\mathcal{F}^T\mapsto \mathbb{R}^T$ is defined by
%the rule
%\begin{equation}
%\label{R_def_int} \left(R^T_if\right)_t=v^f_{1,t}, \quad
%t=1,\ldots,T.
%\end{equation}
%\end{definition}
%The corresponding \emph{response vector} we denote by
%$(r^i_1,r^i_2,\ldots)$. More information on this operator and on
%the inverse spectral problem one can find in the last section.

We introduce the \emph{inner space} of dynamical system
(\ref{Jacobi_dyn}) $\mathcal{H}^T:=\mathbb{R}^T$, $h\in
\mathcal{H}^T$, $h=(h_1,\ldots, h_T)$. The \emph{control operator}
$W^T:\mathcal{F}^T\mapsto \mathcal{H}^T$ is defined by the rule
\begin{equation*}
W^Tf:=u^f_{n,T},\quad n=1,\ldots,T.
\end{equation*}
Directly from (\ref{Jac_sol_rep}) we deduce that
\begin{equation}
\label{W^T_rep} \left(W^Tf\right)_n=u^f_{n,T}=\prod_{k=0}^{n-1}
a_kf_{T-n}+\sum_{s=n}^{T-1}w_{n,s}f_{T-s-1},\quad n=1,\ldots,T.
\end{equation}
The following statement is equivalent to the controllability of
(\ref{Jacobi_dyn}).
\begin{theorem}
\label{teor_control} The operator $W^T$ is an isomorphism between
$\mathcal{F}^T$ and $\mathcal{H}^T$.
\end{theorem}
\begin{proof}
We fix some $a\in \mathcal{H}^T$ and look for a control $f\in
\mathcal{F}^T$ such that $W^Tf=a$. To this aim we write down the
operator as
\begin{equation}
\label{WT}
W^Tf=\begin{pmatrix} u_{1,T}\\
u_{2,T}\\
\cdot\\
u_{k,T}\\
\cdot\\
u_{T,T}
\end{pmatrix}=\begin{pmatrix}
a_0& w_{1,1} & w_{1,2} & \ldots & \ldots & w_{1,T-1}\\
0 & a_0a_1 & w_{2,2} & \ldots & \ldots & w_{2, T-1}\\
\cdot & \cdot & \cdot & \cdot & \cdot & \cdot \\
0 & \ldots & \prod_{j=0}^{k-1}
a_j& w_{k,k} & \ldots & w_{k, T-1}\\
\cdot & \cdot & \cdot & \cdot & \cdot & \cdot \\
0 & 0 & 0 & 0 & \ldots & \prod_{k=0}^{T-1} a_{T-1}
\end{pmatrix}
\begin{pmatrix} f_{T-1}\\
f_{T-2}\\
\cdot\\
f_{T-k-1}\\
\cdot\\
f_{0}
\end{pmatrix}
\end{equation}
We introduce the notations
\begin{eqnarray*}
J_T: \mathcal{F}^T\mapsto \mathcal{F}^T,\quad
\left(J_Tf\right)_n=f_{T-1-n},\quad n=0,\ldots, T-1,\\
A\in \mathbb{R}^{T\times T},\quad
a_{ii}=\prod_{k=0}^{i-1}a_i,\quad a_{ij}=0,\, i\not=j,\\
\\ K\in \mathbb{R}^{T\times T},\quad k_{ij}=0,\, i\geqslant j,
\, k_{ij}=w_{ij-1},\,i<j.
\end{eqnarray*}
Then
\begin{equation}
\label{WT1}
W^T=\left(A+K\right)J^T
\end{equation}
Obviously, this operator is invertible, which proves the statement
of the theorem.
\end{proof}

%For the system (\ref{Jacobi_dyn_int}) the \emph{control operator}
%$W^T_i:\mathcal{F}^T\mapsto \mathcal{H}^N$ is defined by the rule
%\begin{equation*}
%W^T_if:=v^f_{n,T},\quad n=1,\ldots,N.
%\end{equation*}
%he representation for this operator immediately follows from
%(\ref{Jac_sol_rep_int}), (\ref{c_koeff_repr}).

For the system (\ref{Jacobi_dyn}) we introduce the
\emph{connecting operator} $C^T: \mathcal{F}^T\mapsto
\mathcal{F}^T$ by the quadratic form: for arbitrary $f,g\in
\mathcal{F}^T$ we define
\begin{equation}
\label{C_T_def} \left(C^T
f,g\right)_{\mathcal{F}^T}=\left(u^f_{\cdot,T},
u^g_{\cdot,T}\right)_{\mathcal{H}^T}=\left(W^Tf,W^Tg\right)_{\mathcal{H}^T}.
\end{equation}
We observe that $C^T=\left(W^T\right)^*W^T$, so due to Theorem
\ref{teor_control}, $C^T$ is an isomorphism in $\mathcal{F}^T$.
The fact that $C^T$ can be expressed in terms of response $R^{2T-1}$
is crucial in BC-method.
\begin{theorem}
Connecting operator admits the representation in terms of inverse
data:
\begin{equation}
\label{C_T_repr} C^T=a_0C^T_{ij},\quad
C^T_{ij}=\sum_{k=0}^{T-\max{i,j}}r_{|i-j|+2k},\quad r_0=a_0.
\end{equation}
\begin{equation*}
C^T=
\begin{pmatrix}
r_0+r_2+\ldots+r_{2T-2} & r_1+r_3+\ldots+r_{2T-3} & \ldots &
r_T+r_{T-2} &
r_{T-1}\\
r_1+r_3+\ldots+r_{2T-3} & r_0+r_2+\ldots+r_{2T-4} & \ldots &
\ldots
&r_{T-2}\\
\cdot & \cdot & \cdot & \cdot & \cdot \\
r_{T-3}+r_{T-1}+r_{T+1} &\ldots & r_0+r_2+r_4 & r_1+r_3 & r_2\\
r_{T}+r_{T-2}&\ldots &r_1+r_3&r_0+r_2&r_1 \\
r_{T-1}& r_{T-2}& \ldots & r_1 &r_0
\end{pmatrix}
\end{equation*}
\end{theorem}
\begin{proof}
For fixed $f,g\in \mathcal{F}^T$ we introduce the
\emph{Blagoveshchenskii function} by the rule
\begin{equation*}
\psi_{n,t}:=\left(u^f_{\cdot,n},
u^g_{\cdot,t}\right)_{\mathcal{H}^T}=\sum_{k=1}^T
u^f_{k,n}u^g_{k,t}
\end{equation*}
Then we show that $\psi_{n,t}$ satisfies some difference equation.
Indeed, we can evaluate:
\begin{eqnarray*}
\psi_{n,t+1}+\psi_{n,t-1}-\psi_{n+1,t}-\psi_{n-1,t}=\sum_{k=1}^T
u^f_{k,n}\left(u^g_{k,t+1}+u^g_{k,t-1}\right)\\
-\sum_{k=1}^T \left(u^f_{k,n+1}+u^f_{k,n-1}\right)u^g_{k,t}=
\sum_{k=1}^T
u^f_{k,n}\left(a_ku^g_{k+1,t}+a_{k-1}u^g_{k-1,t}+b_ku^g_{k,t}\right)
\\
-\sum_{k=1}^T
\left(a_ku^f_{k+1,n}+a_{k-1}u^f_{k-1,n}+b_ku^f_{k,t}\right)u^g_{k,t}=\sum_{k=1}^T
u^g_{k,t}
\left(a_ku^f_{k+1,n}+a_{k-1}u^f_{k-1,n}\right)\\
+a_0u^g_{0,t}u^f_{1,n}-a_0u^f_{0,n}u^g_{1,t}
+a_Tu^g_{T+1,t}u^f_{T,n}-a_Tu^f_{T+1,n}u^g_{T,t}\\
-\sum_{k=1}^T u^g_{k,t}
\left(a_ku^f_{k+1,n}+a_{k-1}u^f_{k-1,n}\right)=a_0\left[g_t(Rf)_n-f_n(Rg)_t\right]
\end{eqnarray*}
So we arrive at the following difference equation on $\psi_{n,t}$:
\begin{eqnarray}
\label{Blag_eqn}
&\left\{
\begin{array}l
\psi_{n,t+1}+\psi_{n,t-1}-\psi_{n+1,t}-\psi_{n-1,t}=h_{n,t},\quad n,t\in \mathbb{N}_0,\\
\psi_{0,t}=0,\,\, \psi_{n,0}=0,
\end{array}
\right. \\
&h_{n,t}=a_0\left[g_t(Rf)_n-f_n(Rg)_t\right]\notag
\end{eqnarray}
We introduce the set
\begin{eqnarray*}
K(n,t):=\left\{(n,t)\cup \{(n-1,t-1),
(n+1,t-1)\}\cup\{(n-2,t-2),(n,t-2),\right.\\
\left.(n+2,t-2)\}\cup\ldots\cup\{(n-t,0),(n-t+2,0),\ldots,(n+t-2,0),(n+t,0)\}\right\}\\
=\bigcup_{\tau=0}^t\bigcup_{k=0}^\tau
\left(n-\tau+2k,t-\tau\right).
\end{eqnarray*}
The solution to (\ref{Blag_eqn}) is given by (see \cite{MM})
\begin{equation*}
%\label{Blag_sol}
\psi_{n,t}=\sum_{k,\tau\in K(n,t-1)}h(k,\tau).
\end{equation*}
We observe that $\psi_{T,T}=\left(C^Tf,g\right)$, so
\begin{equation}
\label{C_T_sol}\left(C^Tf,g\right)=\sum_{k,\tau\in
K(T,T-1)}h(k,\tau).
\end{equation}
Notice that in the r.h.s. of (\ref{C_T_sol}) the argument $k$ runs
from $1$ to $2T-1.$ We extend $f\in \mathcal{F}^T$,
$f=(f_0,\ldots,f_{T-1})$ to $f\in \mathcal{F}^{2T}$ by:
\begin{equation*}
f_T=0,\quad f_{T+k}=-f_{T-k},\,\, k=1,2,\ldots, T-1.
\end{equation*}
Due to this odd extension, $\sum_{k,\tau\in K(T,T-1)}
f_k(R^Tg)_\tau=0$, so (\ref{C_T_sol}) gives
\begin{eqnarray*}
\left(C^Tf,g\right)=\sum_{k,\tau\in K(T,T-1)}
g_\tau\left(R^{2T}f\right)_k=g_0\left[\left(R^{2T}f\right)_1+\left(R^{2T}f\right)_3+\ldots+\left(R^{2T}f\right)_{2T-1}\right]\\
+g_1\left[\left(R^{2T}f\right)_2+\left(R^{2T}f\right)_4+\ldots+\left(R^{2T}f\right)_{2T-2}\right]+\ldots+g_{T-1}\left(R^{2T}f\right)_{T}.
\end{eqnarray*}
Finally we infer that
\begin{equation*}
C^Tf=\left(\left(R^{2T}f\right)_1+\ldots+\left(R^{2T}f\right)_{2T-1},\left(R^{2T}f\right)_2+\ldots+\left(R^{2T}f\right)_{2T-2},\ldots,\left(R^{2T}f\right)_{T}
\right)
\end{equation*}
from where the statement of the theorem follows.
\end{proof}
%One can observe that $C^T_{ij}$ satisfies the difference boundary
%problem.
%\begin{corollary}
%The kernel of $C^T$ satisfy
%\begin{equation*}
%\left\{
%\begin{array}l
%C^T_{i,j+1}+C^T_{i,j-1}-C^T_{i+1,j}-C^T_{i-1,j}=0,\\
%C^T_{i,T}=r_{T-i},\,\,C^T_{T,j}=r_{T-j},\,\, r_0=1.
%\end{array}
%\right.
%\end{equation*}
%\end{corollary}

%For the system (\ref{Jacobi_dyn_int}) the \emph{connecting
%operator} $C^T_i: \mathcal{F}^T\mapsto \mathcal{F}^T$ is
%introduced in the similar way: for arbitrary $f,g\in
%\mathcal{F}^T$ we define
%\begin{equation}
%\label{C_T_def_int} \left(C^T_i
%f,g\right)_{\mathcal{F}^T}=\left(v^f_{\cdot,T},
%v^g_{\cdot,T}\right)_{\mathcal{H}^N}=\left(W^T_if,W^T_ig\right)_{\mathcal{H}^N}.
%\end{equation}
%More information on $C^T_i$ one can find in the final section.

\section{Inverse problem. }

The dependence of the solution (\ref{Jacobi_dyn}) $u^f$ on the
coefficients $a_n,b_n$ resemble one of the wave equation with the
potential. From the very system one can see that for $M\in
\mathbb{N}$, $u^f_{M,M+1}$ depends on $\{a_0,\ldots,a_{M-1}\}$,
$\{b_1,\ldots,b_M\}$, which implies that $u^f_{1,2M}$ depends of
the same set of parameters. From where follows
\begin{remark}
\label{Rem1} The response $R^{2T}$ (or, what is equivalent, the
response vector $(r_0,r_1,\ldots,r_{2T-1})$) depends on
$\{a_0,\ldots,a_{T-1}\}$, $\{b_1,\ldots,b_T\}$.
\end{remark}
This is an analog of the effect of the finite speed of wave
propagation in the wave equation. This leads to the following
natural set up of the dynamical inverse problem: by the given
operator $R^{2T}$ to recover $\{a_0,\ldots,a_{T-1}\}$ and
$\{b_1,\ldots,b_T\}$.

\subsection{Krein equations}
Let $\alpha,\beta\in \mathbb{R}$ and $y$ be solution to
\begin{equation}
\label{y_special} \left\{
\begin{array}l
a_ky_{k+1}+a_{k-1}y_{k-1}+b_ky_k=0,\\
y_0=\alpha,\,\, y_1=\beta.
\end{array}
\right.
\end{equation}
We set up the following control problem: to find a control $f^T\in
\mathcal{F}^T$ such that
\begin{equation}
\label{Control_probl}
\left(W^Tf^T\right)_k=y_k,\quad
k=1,\ldots,T.
\end{equation}
Due to Theorem \ref{teor_control}, this problem has unique
solution. Let $\varkappa^T$ be a solution to
\begin{equation}
\label{kappa} \left\{
\begin{array}l
\varkappa^T_{t+1}+\varkappa^T_{t-1}=0,\quad t=0,\ldots,T,\\
\varkappa^T_{T}=0,\,\, \varkappa^T_{T-1}=1.
\end{array}
\right.
\end{equation}
We show that the control $f^T$ satisfies the Krein equation:
\begin{theorem}
The control $f^T$, defined by (\ref{Control_probl}) satisfies the
following equation in $\mathcal{F}^T$:
\begin{equation}
\label{C_T_Krein} C^Tf^T=a_0\left[\beta\varkappa^T-\alpha
\left(R^T\right)^*\varkappa^T\right].
\end{equation}
\end{theorem}
\begin{proof}
Let us take $f^T$ solving (\ref{Control_probl}). We observe that
for any fixed $g\in \mathcal{F}^T$:
\begin{equation}
\label{Kr_1}
u^g_{k,T}=\sum_{t=1}^{T-1}\left(u^g_{k,t+1}+u^g_{k,t-1}\right)\varkappa^T_t.
\end{equation}
Indeed, changing the order of summation in the r.h.s. of
(\ref{Kr_1}), we get
\begin{equation*}
\sum_{t=1}^{T-1}\left(u^g_{k,t+1}+u^g_{k,t-1}\right)\varkappa^T_t=\sum_{t=1}^{T-1}\left(\varkappa^T_{t+1}+\varkappa^T_{t-1}
\right)u^g_{k,t}+u^g_{k,0}\varkappa^T_1-u^g_{k,T}\varkappa^T_{T-1}.
\end{equation*}
which gives (\ref{Kr_1}) due to (\ref{kappa}). Using this
observation, we can evaluate
\begin{eqnarray*}
\left(C^Tf^T,g\right)=\sum_{k=1}^T
y_ku^g_{k,T}=\sum_{k=1}^T\sum_{t=0}^{T-1}\left(u^g_{k,t+1}+u^g_{k,t-1}\right)\varkappa^T_t
y_k\\
=\sum_{t=0}^{T-1}\varkappa^T_t\left(\sum_{k=1}^T
\left(a_ku^g_{k+1,t}y_k+a_{k-1}u^g_{k-1,t}y_k + b_ku^g_{k,t}y_k\right)\right)\\
=\sum_{t=0}^{T-1}\varkappa^T_t\left(\sum_{k=1}^T
\left(u^g_{k,t}(a_ky_{k+1}+a_{k-1}y_{k-1}+b_ky_k\right)+a_0u^g_{0,t}y_1\right.\\
\left.+a_Tu^g_{T+1,t}y_T-a_0u^g_{1,t}y_0-a_Tu^g_{T,t}y_{T+1}
\right) =\sum_{t=0}^{T-1}\varkappa^T_t\left(a_0\beta
g_t-a_0\alpha\left(R^Tg\right)_t \right)\\
=\left(\varkappa^T,
a_0\left[\beta g-\alpha
\left(R^Tg\right)\right]\right)=\left(a_0\left[\beta\varkappa^T -
\alpha \left(\left(R^T\right)^*\varkappa^T\right)\right],
g\right).
\end{eqnarray*}
From where (\ref{C_T_Krein}) follows.
\end{proof}

Having found $f^\tau\in W^\tau$ for $\tau=1,\ldots,T$, we can
recover $b_n,$ $a_n$, $n=1,\ldots,T-1$. We will describe the
procedure. From (\ref{Jac_sol_rep}) and (\ref{Goursat}) we infer
that
\begin{eqnarray*}
u^{f^T}_{T,T}=\prod_{k=0}^{T-1}a_k f^T_0,\\
u^{f^T}_{T-1,T}=\prod_{k=0}^{T-2}a_k
f^T_1+\prod_{k=0}^{T-2}(b_1+b_2+\ldots+b_{T-1})
\end{eqnarray*}
Notice that we know $a_0=r_0$. Let $T=2$, then we have:
\begin{eqnarray}
y_2=u^{f^2}_{2,2}=a_0a_1f^2_0,\label{y_2}\\
y_1=u^{f^2}_{1,2}=a_0f^2_1+a_0b_1f^2_0, \label{y_1}
\end{eqnarray}
In (\ref{y_1}) we know $y_1=\beta,$ $a_0,$ $f^2_1,$ $f^2_0$, so we
can recover $b_1$. On the other hand, using (\ref{y_special}), we
have a system
\begin{equation*}
\left\{
\begin{array}l
y_2=a_0a_1f^2_0,\\
a_1y_2+a_0\alpha+b_1\beta=0
\end{array}
\right.
\end{equation*}
Since $a_1>0,$ we can recover $y_2$ and $a_1.$ Assume that we have
already found $y_{k-1},$ $b_{k-2},$ $a_{k-2}$ for $k\leqslant n$,
we will find $y_n,$ $a_{n-1},$ $b_{n-1}$. We have that
\begin{eqnarray}
y_n=u^{f^n}_{n,n}=\prod_{k=0}^{n-2}a_k a_{n-1}f^2_0,\label{y_n}\\
y_{n-1}=u^{f^n}_{n-1,n}=\prod_{k=0}^{n-2}a_kf^n_1+\prod_{k=0}^{n-2}a_k(b_1+\ldots+b_{n-2}+b_{n-1})f^n_0,
\label{y_n1}
\end{eqnarray}
Since we know $y_{n-1},$ $f^n_0$, $f^n_1$, and $a_k,$ $b_k$,
$k\leqslant n-2,$ from (\ref{y_n1}) we can recover $b_{n-1}$. Then
we use (\ref{y_special}) and (\ref{y_n}) to write down the system
\begin{equation*}
\left\{
\begin{array}l
y_n=\prod_{k=0}^{n-2}a_k a_{n-1}f^2_0,\\
a_{n-1}y_n+a_{n-2}y_{n-2}+b_{n-1}y_{n-1}=0.
\end{array}
\right.
\end{equation*}
From which we recover $a_{n-1}$ and $y_n.$

\subsection{Factorization method}

We make use the fact that matrix $C^T$ has a special structure --
it is a product of triangular matrix and its conjugate. We rewrite
the operator $W^T$ as $W^T=\overline W^TJ$ where
\begin{equation*}
W^Tf=\begin{pmatrix}
a_0 & w_{1,1} & w_{1,2} & \ldots & w_{1,T-1}\\
0 & a_0a_1 & w_{2,2} &  \ldots & w_{2, T-1}\\
\cdot & \cdot & \cdot & \cdot & \cdot \\
0 & \ldots & \prod_{j=1}^{k-1}a_j& \ldots & w_{k, T-1}\\
\cdot & \cdot & \cdot  & \cdot & \cdot \\
0 & 0 & 0  & \ldots & \prod_{j=1}^{T-1}a_j
\end{pmatrix}
\begin{pmatrix}
0 & 0 & 0 & \ldots & 1\\
0 & 0 & 0 & \ldots  & 0\\
\cdot & \cdot & \cdot & \cdot &  \cdot \\
0 & \ldots & 1& 0  & 0\\
\cdot & \cdot & \cdot & \cdot &  \cdot \\
1 & 0 & 0 & 0 &  0
\end{pmatrix}
\begin{pmatrix} f_{0}\\
f_{2}\\
\cdot\\
f_{T-k-1}\\
\cdot\\
f_{T-1}
\end{pmatrix}
\end{equation*}
Using the definition (\ref{C_T_def}) and the invertibility of
$W^T$ (cf. Theorem \ref{teor_control}), we have:
\begin{equation*}
C^T=\left(W^T\right)^*W^T,\quad \text{or} \quad
\left(\left(W^T\right)^{-1}\right)^*C^T\left(W^T\right)^{-1}=I.
\end{equation*}
We can rewrite the latter equation as
\begin{equation}
\label{C_T_eqn_ker} \left(\left(\overline
W^T\right)^{-1}\right)^*\overline C^T\left(\overline
W^T\right)^{-1}=I,\quad \overline C^T=JC^TJ.
\end{equation}
Here the matrix $\overline C^T$ has the entries:
\begin{equation}
\label{C_overline_repr} \overline C_{ij}=C_{T+1-j,T+1-i},\quad
\overline C^T=a_0
\begin{pmatrix}
r_0 & r_1 & r_2 & \ldots & r_{T-1}\\
r_1 & r_0+r_2 & r_1+r_3 & \ldots
&..\\
r_2 & r_1+r_3 & r_0+r_2+r_4 & \ldots & ..\\
\cdot & \cdot & \cdot & \cdot & \cdot \\
\end{pmatrix},
\end{equation}
and operator $\left(\overline W^T\right)^{-1}$ has the form
\begin{equation}
\label{W_T_bar} \left(\overline W^T\right)^{-1}=\begin{pmatrix}
a_{1,1} & a_{1,2} & a_{1,3}& \ldots & a_{1,T} \\
0 & a_{2,2} & a_{2,3} &\ldots &..\\
\cdot & \cdot & \cdot & a_{T-1,T-1} & a_{T-1,T} \\
0 &\ldots &\ldots & 0 & a_{T,T}
\end{pmatrix},
\end{equation}
We multiply the $k-$th row of $\overline W^T$ by $k-$th column of
$\left(\overline W^T\right)^{-1}$ to get $a_{k,k}a_0a_1\ldots
a_{k-1}=1$, so
\begin{equation}
\label{AK_form} a_{k,k}=\left(\prod_{j=0}^{k-1}a_j\right)^{-1}.
\end{equation}
Multiplying the $k-$th row of $\overline W^T$ by $k+1-$th column
of $\left(\overline W^T\right)^{-1}$, we get
\begin{equation*}
a_{k,k+1}a_0a_1\ldots a_{k-1}+a_{k+1,k+1}w_{k,k}=0,
\end{equation*}
from where
\begin{equation}
\label{AK1_form}
a_{k,k+1}=-\left(\prod_{j=0}^{k}a_j\right)^{-2}a_k w_{k,k},
\end{equation}

Thus we can rewrite (\ref{C_T_eqn_ker}) as
\begin{equation}
\label{matr_eq}
\begin{pmatrix}
a_{1,1} & 0 & . &  0 \\
a_{1,2} & a_{2,2} & 0  &.\\
\cdot & \cdot & \cdot & \cdot  \\
a_{1,T} & . & .  & a_{T,T}
\end{pmatrix}
\begin{pmatrix}
\overline c_{11} & .. & .. &  \overline c_{1T} \\
.. & .. & ..  &..\\
\cdot & \cdot & \cdot & \cdot  \\
\overline c_{T1} &.. &   & \overline c_{TT}
\end{pmatrix}\begin{pmatrix}
a_{1,1} & a_{1,2} & ..& a_{1,T} \\
0 & a_{2,2} & ..  & a_{2,T}\\
\cdot & \cdot & \cdot & \cdot  \\
0 &\ldots &\ldots  & a_{T,T}
\end{pmatrix}=I
%\begin{pmatrix}
%1 & 0 & .. &  0 \\
%0 & 1 & ..  &0\\
%\cdot & \cdot & \cdot & \cdot  \\
%0 &0 &  . & 1
%\end{pmatrix}
\end{equation}
In the above equation $\overline c_{ij}$ are given (see
(\ref{C_overline_repr})), the entries $a_{ij}$ are unknown. As a
direct consequence of (\ref{matr_eq}) we get
\begin{equation*}
\det{A^T}\det{\overline C^T}\det{A^T}=1,
\end{equation*}
which yields
\begin{equation*}
a_{1,1}*\ldots*a_{T,T}=\left(\det{\overline
C^T}\right)^{-\frac{1}{2}}.
\end{equation*}
From where we derive that
\begin{equation}
a_{1,1}=\left(\det{\overline C^1}\right)^{-\frac{1}{2}},\quad
a_{2,2}=\left(\frac{\det{\overline C^2}}{\det{\overline
C^1}}\right)^{-\frac{1}{2}},\quad
a_{k,k}=\left(\frac{\det{\overline C^k}}{\det{\overline
C^{k-1}}}\right)^{-\frac{1}{2}}.
\end{equation}
Combining the latter equation with (\ref{AK_form}), we deduce
\begin{equation}
a_0*\ldots*a_{k-1}=\left(\frac{\det{\overline C^k}}{\det{\overline
C^{k-1}}}\right)^{\frac{1}{2}},
\end{equation}
similarly,
\begin{equation}
a_0*\ldots*a_{k}=\left(\frac{\det{\overline
C^{k+1}}}{\det{\overline C^{k}}}\right)^{\frac{1}{2}},
\end{equation}
so we can write
\begin{equation}
\label{AK} a_k=\frac{\left(\det{\overline
C^{k+1}}\right)^{\frac{1}{2}}\left(\det{\overline
C^{k-1}}\right)^{\frac{1}{2}}}{\det{\overline C^{k}}}
\end{equation}
here we assume that $\det{C^0}=1,$ $\det{C^{-1}=1}$.

Now using (\ref{matr_eq}) we can write down the equation on the
last column of $\left(\overline W^T\right)^{-1}$:
\begin{equation}
\begin{pmatrix}
a_{1,1} & 0 & . &  0 \\
a_{1,2} & a_{2,2} & 0  &.\\
\cdot & \cdot & \cdot & \cdot  \\
a_{1,T-1} & . & .  & a_{T-1,T-1}
\end{pmatrix}
\begin{pmatrix}
\overline c_{1,1} & .. & .. &  \overline c_{1,T} \\
.. & .. & ..  &..\\
\cdot & \cdot & \cdot & \cdot  \\
\overline c_{T-1,1} &.. &   & \overline c_{T-1,T-1}
\end{pmatrix}\begin{pmatrix}
a_{1,T} \\
a_{2,T}\\
\cdot  \\
a_{T,T}
\end{pmatrix}=\begin{pmatrix}
0 \\
0\\
\cdot  \\
0
\end{pmatrix}
\end{equation}
Here we know $a_{T,T}$, so $(a_{1,T},\ldots,a_{T-1,T})^*$
satisfies
\begin{eqnarray*}
\begin{pmatrix}
a_{1,1} & 0 & . &  0 \\
a_{1,2} & a_{2,2} & 0  &.\\
\cdot & \cdot & \cdot & \cdot  \\
a_{1,T-1} & . & .  & a_{T-1,T-1}
\end{pmatrix}
\begin{pmatrix}
\overline c_{1,1} & .. & .. &  \overline c_{1,T} \\
.. & .. & ..  &..\\
\cdot & \cdot & \cdot & \cdot  \\
\overline c_{T-1,1} &.. &   & \overline c_{T-1,T-1}
\end{pmatrix}\begin{pmatrix}
a_{1,T} \\
a_{2,T}\\
\cdot  \\
a_{T-1,T}
\end{pmatrix}\\
+a_{T,T}\begin{pmatrix}
a_{1,1} & 0 & . &  0 \\
a_{1,2} & a_{2,2} & 0  &.\\
\cdot & \cdot & \cdot & \cdot  \\
a_{1,T-1} & . & .  & a_{T-1,T-1}
\end{pmatrix}\begin{pmatrix}
a_{1,T} \\
a_{2,T}\\
\cdot  \\
a_{T-1,T}
\end{pmatrix}=\begin{pmatrix}
0 \\
0\\
\cdot  \\
0
\end{pmatrix}
\end{eqnarray*}
which is equivalent to
\begin{equation}
\label{matr_eq2}
\begin{pmatrix}
\overline c_{1,1} & .. & .. &  \overline c_{1,T} \\
.. & .. & ..  &..\\
\cdot & \cdot & \cdot & \cdot  \\
\overline c_{T-1,1} &.. &   & \overline c_{T-1,T-1}
\end{pmatrix}\begin{pmatrix}
a_{1,T} \\
a_{2,T}\\
\cdot  \\
a_{T-1,T}
\end{pmatrix}=-a_{T,T}\begin{pmatrix}
a_{1,T} \\
a_{2,T}\\
\cdot  \\
a_{T-1,T}
\end{pmatrix}
\end{equation}
Introduce the notation:
\begin{equation}
\overline C^{k-1}_k:=\begin{pmatrix}
\overline c_{1,1} & .. & .. &  \overline c_{1,k-2} &  \overline c_{1,k}\\
.. & .. & ..  &..\\
\cdot & \cdot & \cdot & \cdot  \\
\overline c_{k-1,1} &.. &   & \overline c_{k-1,k-2}& \overline
c_{k-1,k}
\end{pmatrix},
\end{equation}
that is $\overline C^{k-1}_k$ is constructed from $C^{k-1}$ by
substituting the last column by $(\overline
c_{1,k},\ldots,\overline c_{k-1,k})$. Then by linear algebra, from
(\ref{matr_eq2}) we have:
\begin{equation}
\label{a1} a_{T-1,T}=-a_{T,T}\frac{\det{\overline
C^{T-1}_T}}{\det{\overline C^{T-1}}},
\end{equation}
here we assume that $\det{C^{-1}_0}=0.$ On the other hand, from
(\ref{AK_form}), (\ref{AK1_form}) we see that
\begin{equation}
\label{a2}
a_{T-1,T}=\left(\prod_{j=0}^{T-1}a_j\right)^{-1}\sum_{k=1}^{T-1}b_k
\end{equation}
Equating (\ref{a1}) and (\ref{a2}), we see that
\begin{equation}
\sum_{k=1}^{T-1}b_k=-\frac{\det{\overline
C^{T-1}_T}}{\det{\overline C^{T-1}}},\quad
\sum_{k=1}^{T}b_k=-\frac{\det{\overline
C^{T}_{T+1}}}{\det{\overline C^{T}}},
\end{equation}
from where
\begin{equation}
\label{BK} b_k=-\frac{\det{\overline C^{k}_{k+1}}}{\det{\overline
C^{k}}}+\frac{\det{\overline C^{k-1}_k}}{\det{\overline C^{k-1}}}
\end{equation}

\subsection{Characterization of the inverse data.}

In the second section we considered the forward problem
(\ref{Jacobi_dyn}), for $(a_0,\dots, a_{T-1})$, $(b_1,\dots,
b_{T-1})$ we constructed the matrix $W^T$ (\ref{Jac_sol_rep}),
(\ref{Goursat}), the response vector $(r_0,r_1,\dots,r_{2T-2})$
(see (\ref{R_def})) and the connecting operator  $\overline C^T$
defined in (\ref{C_T_repr}), (\ref{C_overline_repr}). From the
theorem \ref{teor_control} we know that $C^T$ is positively
definite. We have also shown that if coefficients $r_0,
r_1,\dots,r_{2T-2}$ correspond to $(a_0,\dots,a_{T-1})$
$(b_1,\dots,b_{T-1})$ then we can recover those $a'$s and $b's$ by
(\ref{AK}) and (\ref{BK}).

Now we set up a question: can one determine whether a vector
$(r_0,r_1,r_2,\ldots,r_{2T-2})$ is a response vector for the
dynamical system (\ref{Jacobi_dyn}) with some
$(a_0,\ldots,a_{T-1})$ $(b_1,\ldots,b_{T-1})$ or not? The answer
is the following theorem.
\begin{theorem}
\label{Th_char} The  vector $(r_0,r_1,r_2,\ldots,r_{2T-2})$ is a
response vector for the dynamical system (\ref{Jacobi_dyn}) if and
only if the matrix $C^T$ (\ref{C_T_repr}) is positively definite.
\end{theorem}
\begin{proof}
First we observe that in the conditions of the theorem we can
substitute $C^T$ by $\overline C^T$ (\ref{C_overline_repr}). The
necessary part of the theorem is proved in the preceding sections.
We are left to prove the sufficiency of these conditions.

Let we have a vector $(r_0,r_1,\dots,r_{2T-2})$ such that the
matrix $\overline C^{T}$ constructed from it using
(\ref{C_overline_repr}) satisfies conditions of the theorem. Then
we can construct sequences $(a_0,\dots,a_{T-1})$,
$(b_1,\dots,b_{T-1})$ using (\ref{AK}) and (\ref{BK}) and consider
the dynamical system (\ref{Jacobi_dyn}) with this coefficients.
For this system we construct the response
$(r^{new}_0,r^{new}_1,\dots,r^{new}_{2T-2})$ and connecting
operator $\mathcal{C}^{T}$ and its rotated $ \mathcal{\overline
C}^{T}$ using (\ref{C_T_repr}) and (\ref{C_overline_repr}). We
will show that the response vectors coincide.

First of all we note that we have two matrices constructed by
(\ref{C_overline_repr}), one comes from the vector
$(r_0,r_1,\dots,r_{2T-2})$ and the other comes from
$(r_0^{new},r^{new}_1,\dots,r^{new}_{2T-2})$. Also they have a
common property that $\overline C^T>0,$ $\mathcal{\overline
C}^T>0$ (one by theorem's condition and the other by
representation $\mathcal{\overline C}^T= (\overline W^T_{new})^*
\overline W^T_{new}$). Secondly we note that if we calculate the
elements of sequences  $(a_0,\dots,a_{T-1})$,
$(b_1,\dots,b_{T-1})$ using (\ref{AK}) and (\ref{BK}) from any of
$\overline C^{T}$ and $\mathcal{\overline C}^{T}$ matrices, we get
the same answer. If so, we get that
\begin{eqnarray*}
\frac{\det{\overline C^{k-1}_k}}{\det{\overline
C^{k-1}}}-\frac{\det{\overline C^{k}_{k+1}}}{\det{\overline
C^{k}}}=\frac{\det{\mathcal{\overline
C}^{k-1}_k}}{\det{\mathcal{\overline
C}^{k-1}}}-\frac{\det{\mathcal{\overline
C}^{k}_{k+1}}}{\det{\mathcal{\overline C}^{k}}},\\
\frac{\left(\det{\overline
C^{k+1}}\right)^{\frac{1}{2}}\left(\det{\overline
C^{k-1}}\right)^{\frac{1}{2}}}{\det{\overline
C^{k}}}=\frac{\left(\det{\mathcal{\overline
C}^{k+1}}\right)^{\frac{1}{2}}\left(\det{\mathcal{\overline
C}^{k-1}}\right)^{\frac{1}{2}}}{\det{\mathcal{\overline
C}^{k}}},\\
\det{\overline C^0}=\det{\mathcal{\overline C}^0}=1, \quad
\det{\overline C^{-1}}=\det{\mathcal{\overline C}^{-1}}=1,\\
\det{\overline C^{-1}_0}=\det{\mathcal{\overline C}^{-1}_0}=0,
\end{eqnarray*}
From these equalities by simple arguments we deduce that
\begin{eqnarray*}
\det{\overline C^k}=\det{\mathcal{\overline C}^k},\\
\det{\overline C^k_{k+1}}=\det{\mathcal{\overline C}^k_{k+1}}.
\end{eqnarray*}
From these equalities immediately follows that
\begin{equation*}
r_k=r_k^{new}, \quad k=1,\ldots, 2T-2.
\end{equation*}
which finishes the proof.

\end{proof}

\subsection{Discrete Schr\"odinger operator}

Here we consider the case of the dynamical Schr\"odinger operator,
i.e. the system (\ref{Jacobi_dyn}) with $a_k=1$, $k\in
\mathbb{N}$, see\cite{MM}. In this particular case the control
operator (\ref{WT}) is given by $W^T=\left(I+K\right)J^T$, so all
the diagonal elements of the matrix in (\ref{WT}) are equal to
one. The latter immediately yields $\det{W^T}=1$. Due to this
fact, the connecting operator (\ref{C_T_def}), (\ref{C_T_repr}),
has a remarkable property that $\det{C^k}=1$, $k=1,\ldots,T.$ This
fact actually says that not all elements in the response vector
are independent: $r_{2m}$ depends on $r_{2l+1}$, $l=0,\ldots,m-1$,
moreover, this property characterize the dynamical data of the
discrete Schr\"odinger operators:
\begin{theorem}
\label{Teor_Sch} The  vector $(1,r_1,r_2,\ldots,r_{2T-2})$ is a
response vector for the dynamical system (\ref{Jacobi_dyn}) with
$a_k=1$ if and only if the matrix $C^T$ (\ref{C_T_repr}) is
positive definite and $\det C^l=1,$ $l=1,\ldots,T$.
\end{theorem}
\begin{proof}
As in  Theorem \ref{Th_char} we use $\overline C^T$ instead of
$C^T.$ The necessity of the conditions was explained. We are left
with the sufficiency part.

Notice that $r_0=a_0=1$. Let a vector $(1,r_1,\dots,r_{2T-2})$ be
such that the matrix $\overline C^{T}$ constructed from it using
(\ref{C_overline_repr}) satisfies conditions of the theorem. We
construct the potential $(b_1,\dots,b_{T-1})$ using (\ref{BK}) and
consider the dynamical system (\ref{Jacobi_dyn}) with these $b_k$
and $a_k=1$. For this system we construct the response
$(1,r_1^{new},\ldots,r_{2T-2}^{new})$ and the connecting operator
$\mathcal{\overline C}^T$ using (\ref{R_def}), (\ref{C_T_repr})
and (\ref{C_overline_repr}). We will show that responses coincide.

We notice that if we calculate $(b_1,\dots,b_{T-1})$ using
(\ref{BK}) with any of $\overline C^{T}$ or $\mathcal{\overline
C}^{T}$ matrices, we get the same answer. The latter implies (we
count that $\det{\overline C^k}=\det{\mathcal{\overline C}^k}=1$)
\begin{eqnarray*}
\det{\overline C^{k-1}_k}-\det{\overline
C^{k}_{k+1}}=\det{\mathcal{\overline
C}^{k-1}_k}-\det{\mathcal{\overline C}^{k}_{k+1}},\\
\det{\overline C^{-1}_0}=\det{\mathcal{\overline C}^{-1}_0}=0
\end{eqnarray*}
By induction arguments we get
\begin{eqnarray*}
\det{\overline C^{k}_{k+1}}=\det{\mathcal{\overline
C}^{k}_{k+1}},\\
\det{\overline C^k}=\det{\mathcal{\overline C}^k}=1,
\end{eqnarray*}
which yields $r_k=r_k^{new}$, $k=1,\ldots, 2T-2$. That finishes
the proof.
\end{proof}

\section{Spectral representation of $C^T$ and $r_t$.}

We fix $N\in \mathbb{N}$. Along with (\ref{Jacobi_dyn}) we
consider the analog of the wave equation  on the interval: we
impose the Dirichlet condition at $n=N+1$. Then for a control
$f=(f_0,f_1,\ldots)$ and $h\in \mathbb{R}$ we consider
\begin{equation}
\label{Jacobi_dyn_int} \left\{
\begin{array}l
v_{n,t+1}+v_{n,t-1}-a_nv_{n+1,t}-a_{n-1}v_{n-1,t}-b_nv_{n,t}=0,\quad t\in \mathbb{N}_0,\,\, n\in 0,\ldots, N+1\\
v_{n,-1}=v_{n,0}=0,\quad n=1,2,\ldots,N+1 \\
v_{0,t}=f_t,\quad v_{N+1,t}+hv_{N,t}=0,\quad t\in \mathbb{N}_0.
\end{array}\right.
\end{equation}
We denote the solution to (\ref{Jacobi_dyn_int}) by $v^f$.
%The
%response operator is introduced as $\left(\quad R_i^\tau
%f\right)_t:=v^f_{1,t}$

Let $\phi_n(\lambda)$ be the solution to
\begin{equation}
\label{Spec_sol} \left\{
\begin{array}l a_n\phi_{n+1}+a_{n-1}\phi_{n-1}+b_n\phi_n=\lambda\phi_n,\\
\phi_0=0,\,\,\phi_1=1.
\end{array}
\right.
\end{equation}
%We introduce the Hamiltonian
%\begin{equation*}
%H_N:=\begin{pmatrix} b_1 & a_1 & 0 & \ldots & 0\\
%a_1 & b_2 & a_2 & \ldots & 0\\
%\cdot & \cdot &\cdot &\cdot &\cdot \\
%0 &\ldots & 0& a_{N-1}& b_N
%\end{pmatrix}
%\end{equation*}
Denote by $\{\lambda_k\}_{k=1}^N$ the roots of the equation
$\phi_{N+1}+h\phi_N=0$, it is known \cite{A}, that they are real.
We introduce the vectors $\phi^n\in \mathbb{R}^N$ by the rule
$\phi^n_i:=\phi_i(\lambda_n)$, $n,i=1,\ldots,N,$ and define
%Let
%$\{\varphi^k,\lambda_k\}_{k=1}^N$ be eigenvectors chosen such that
%$\varphi^k_1=1$ and eigenvalues of $H_N$. Introduce
the numbers $\rho_k$ by
\begin{equation}
\label{Ortog} (\phi^k,\phi^l)=\delta_{kl}\frac{\rho_k}{a_0},
\end{equation}
where $(\cdot,\cdot)$-- is a scalar product in $\mathbb{R}^N$.
\begin{definition}
The set
\begin{equation}
\label{SP_data} \{\lambda_k,\rho_k\}_{k=1}^N
\end{equation}
is called the spectral data.
\end{definition}
%On introducing vectors $\phi^n\in \mathbb{R}^N$ by the rule
%$\phi^n_i:=\phi_i(\lambda_n)$, $n,i=1,\ldots,N,$ we have
%\begin{proposition}
%The solutions of $\phi_{N+1}(\lambda)=0$ are $\lambda_n$,
%$n=1,\ldots,N$; and $\phi^n_i=\varphi^n_i$, $n,i=1,\ldots,N.$
%\end{proposition}

We take $y\in \mathbb{R}^N, \,\,y=(y_1,\ldots,y_N)$, for each $n$
we multiply the equation in (\ref{Jacobi_dyn_int}) by $y_n$, sum
up and evaluate the following expression, changing the order of
summation
\begin{eqnarray}
0=\sum_{n=1}^N \left(v_{n,t+1}y_n+v_{n,t-1}y_n - a_nv_{n+1,t}y_n-
a_{n-1}v_{n-1,t}y_n-b_nv_{n,t}y_n\right)\notag\\
=\sum_{n=1}^N \left(v_{n,t+1}y_n+v_{n,t-1}y_n -
v_{n,t}(a_{n-1}y_{n-1}+a_ny_{n+1})-b_nv_{n,t}y_n\right)\notag\\
-a_Nv_{N+1,t}y_N-a_0v_{0,t}y_1+a_0v_{1,t}y_0+a_Nv_{N,t}y_{N+1}\label{int_rav}
\end{eqnarray}
Now we choose $y=\phi^l$, $l=1\ldots, N$. On counting that
$\phi^l_0=0$, $\phi^l_{N+1}=-h\phi_{N},$ $\phi^l_1=1,$
$v_{0,t}=f_t$, $v_{N+1,t}+hv_{N,t}=0$ we evaluate (\ref{int_rav})
arriving at:
\begin{equation}
\label{Integr_eqn} 0=\sum_{n=1}^N
\left(v_{n,t+1}\phi^l_n+v_{n,t-1}\phi^l_n -
v_{n,t}\left(a_{n-1}\phi^l_{n-1}+a_n\phi^l_{n+1}+b_n\phi^l_n\right)\right)
-a_0f_t=0
\end{equation}

We assume that the solution to (\ref{Jacobi_dyn_int}) has a form
\begin{equation}
\label{Jac_sol_rep_int} v^f_{n,t}= \left\{
\begin{array}l
\sum_{k=1}^N c_t^k\phi^k_n,\quad n=1,\ldots,N\\
f_t,\quad n=0.
\end{array}
\right.
\end{equation}
\begin{proposition}
The coefficients $c^k$ admits the representation:
\begin{equation}
\label{c_koeff_repr}
c^k=\frac{a_0}{\rho_k}T\left(\lambda_k\right)*f,
\end{equation}
where
$T(2\lambda)=(T_1(2\lambda),T_2(2\lambda),T_3(2\lambda),\ldots)$
are Chebyshev polynomials of the second kind.
\end{proposition}
\begin{proof}
We plug (\ref{Jac_sol_rep_int}) into (\ref{Integr_eqn}) and
evaluate, counting that
$a_{n-1}\phi^l_{n-1}+a_n\phi^l_{n+1}+b_n\phi^l_n=\lambda_l\phi^l_n$:
\begin{eqnarray*}
\sum_{n=1}^N \left(v_{n,t+1}+v_{n,t-1}-\lambda_l
v_{n,t}\right)\phi^l_n =a_0f_t,\\
\sum_{n=1}^N
\sum_{k=1}^N\left(c^k_{t+1}\phi^k_{n}+c^k_{t-1}\phi^k_{n}-\lambda_l
c^k_{t}\phi^k_{n}\right)\phi^l_n =a_0f_t.
\end{eqnarray*}
Changing the order of summation and using (\ref{Ortog}) we finally
arrive at the following  equation on $c^k_t$, $k=1,\ldots, N$:
\begin{equation}
\label{c_eqn} \left\{
\begin{array}l
c^k_{t+1}+c^k_{t-1}-\lambda_k c^k_{t}=\frac{a_0}{\rho_k}f_t,\\
c^k_{-1}=c^k_0=0.
\end{array}
\right.
\end{equation}
We assume that solution to (\ref{c_eqn}) has a form
$c^k=\frac{a_0}{\rho_k}T*f,$ or
\begin{equation}
\label{c_eqn_1} c^k_t=\frac{a_0}{\rho_k}\sum_{l=0}^{t}T_lf_{t-l}.
\end{equation}
Plugging it into (\ref{c_eqn}), we get
\begin{eqnarray*}
\frac{a_0}{\rho_k}\left(\sum_{l=0}^{t+1}f_lT_{t+1-l}+
\sum_{l=0}^{t-1}f_l T_{t-1-l}-\lambda_k\sum_{l=0}^{t}f_l T_{t-l}\right)=\frac{a_0}{\rho_k}f_t\\
\sum_{l=0}^{t}f_l\left(T_{t+1-l}+T_{t-1-l}-\lambda_k
T_{t-l}\right)+f_t T_1-f_{t-1}T_0 =f_t.
\end{eqnarray*}
We see that (\ref{c_eqn_1}) holds if $T$ solves
\begin{equation*}
\left\{
\begin{array}l
T_{t+1}+T_{t-1}-\lambda_k T_{t}=0,\\
T_{0}=0,\,\, T_1=1.
\end{array}
\right.
\end{equation*}
Thus $T_k(2\lambda)$ are Chebyshev polynomials of the second kind.
\end{proof}

For the system (\ref{Jacobi_dyn_int}) the \emph{control operator}
$W^T_{N,h}:\mathcal{F}^T\mapsto \mathcal{H}^N$ is defined by the
rule
\begin{equation*}
W^T_{N,h}f:=v^f_{n,T},\quad n=1,\ldots,N.
\end{equation*}
The representation for this operator immediately follows from
(\ref{Jac_sol_rep_int}), (\ref{c_koeff_repr}). Because of the
dependence of the solution on the coefficients, which was
discussed in the third section, we see that $v^f_{N,N}$ does not
"feel" the boundary condition at $n=N$, so
\begin{equation}
\label{Rav_bc} u^f_{n,t}=v^f_{n,t},\quad n\leqslant t\leqslant
N,\quad \text{and}\quad W^N=W^N_{N,h}.
\end{equation}
We introduce the response operator $R^T_{N,h}:\mathcal{F}^T\mapsto
\mathbb{R}^T$ by the rule
\begin{equation}
\label{R_def_int} \left(R^T_{N,h}f\right)_t=v^f_{1,t}, \quad
t=1,\ldots,T.
\end{equation}
The \emph{connecting operator} $C^T_{N,h}: \mathcal{F}^T\mapsto
\mathcal{F}^T$ is introduced in the similar way: for arbitrary
$f,g\in \mathcal{F}^T$ we define
\begin{equation}
\label{C_T_def_int} \left(C^T_{N,h}
f,g\right)_{\mathcal{F}^T}=\left(v^f_{\cdot,T},
v^g_{\cdot,T}\right)_{\mathcal{H}^N}=\left(W^T_{N,h}f,W^T_{N,h}g\right)_{\mathcal{H}^N}.
\end{equation}

The dependence of the solution (\ref{Jacobi_dyn}) $u^f$ on
$a_n,b_n$ is discussed in the beginning of the section three (see
Remark \ref{Rem1}). This dependence in particular implies that for
$M\in \mathbb{N}$,
\begin{equation}
\label{R_eqv}
R^{2M} = R^{2M}_{M,h}.
\end{equation}
i.e. $v^f_{1,2M}$ does not "feel" the boundary condition at $n=M$.
We introduce the special control $\delta=(1,0,0,\ldots)$, then the
kernel of response operator (\ref{R_def_int}) is
\begin{equation}
\label{con1}
r_{t-1}^{N,h}=\left(R^T_{N,h}\delta\right)_t=v^\delta_{1,t},\quad
t=1,\ldots,T.
\end{equation}
on the other hand, we can use (\ref{Jac_sol_rep_int}),
(\ref{c_koeff_repr}) to obtain:
\begin{equation}
\label{con2}
v^\delta_{1,t}=\sum_{k=1}^N\frac{a_0}{\rho_k}T_t(\lambda_k).
\end{equation}
So on introducing the spectral function
\begin{equation}
\label{Spectr_fun}
\rho^{N,h}(\lambda)=\sum_{\{k\,|\,\lambda_k<\lambda\}}\frac{a_0}{\rho_k},
\end{equation}
from (\ref{con1}), (\ref{con2}) we deduce that
\begin{equation*}
r_{t-1}^{N,h}=\int_{-\infty}^\infty
T_t(\lambda)\,d\rho^{N,h}(\lambda),\quad t\in \mathbb{N}.
\end{equation*}
Due to (\ref{R_eqv}), we get
\begin{equation}
\label{resp_repr_mes} r_{t-1}=r_{t-1}^{N,h}=\int_{-\infty}^\infty
T_t(\lambda)\,d\rho^{N,h}(\lambda),\quad t\in 1,\ldots,2N.
\end{equation}
Taking in (\ref{resp_repr_mes}) $N$ to infinity, and varying $h$,
we come to the
\begin{equation*}
r_{t-1}=\int_{-\infty}^\infty T_t(\lambda)\,d\rho(\lambda),\quad
t\in \mathbb{N},
\end{equation*}
where $d\rho$ is a spectral measure of the operator $H$
(non-unique when $H$ is in the limit-circle case at infinity).

Let us evaluate $(C^T_Nf,g)$ for $f,g\in \mathcal{F}^T$, using the
expansion (\ref{Jac_sol_rep_int}):
\begin{eqnarray*}
(C^T_{N,h}f,g)=\sum_{n=1}^N
v^f_{n,T}v^g_{n,T}=\sum_{n=1}^N\sum_{k=1}^N
\frac{a_0}{\rho_k}T_T\left(\lambda_k\right)*f\varphi^k_n \,
\sum_{l=1}^N
\frac{a_0}{\rho_l}T_T\left(\lambda_l\right)*g\varphi^l_n\\
=\sum_{k=1}^N\frac{a_0}{\rho_k}T_T(\lambda_k)*f
T_T(\lambda_k)*g=\int_{-\infty}^\infty
\sum_{l=0}^{T-1}T_{T-l}(\lambda)f_l
\sum_{m=0}^{T-1}T_{T-m}(\lambda)g_m\,d\rho^{N,h}(\lambda)
\end{eqnarray*}
from the equality above it is evident that (cf. (\ref{C_T_repr}))
\begin{equation}
\label{SP_mes_d} \{C^T_{N,h}\}_{l+1,m+1}=\int_{-\infty}^\infty
T_{T-l}(\lambda)T_{T-m}(\lambda)\,d\rho^{N,h}(\lambda), \quad
l,m=0,\ldots,T-1.
\end{equation}
Taking into account (\ref{Rav_bc}), we obtain that $C^T=C^T_{N,h}$
with $N\geqslant T$, so (\ref{SP_mes_d}) yields for $N\geqslant
T:$
\begin{equation}
\label{SP_mes_d1} \{C^T\}_{l+1,m+1}=\int_{-\infty}^\infty
T_{T-l}(\lambda)T_{T-m}(\lambda)\,d\rho^{N,h}(\lambda), \quad
l,m=0,\ldots,T-1,
\end{equation}
and passing to the limit as $N\to\infty$ yields
\begin{equation}
\label{SP_mes_d2} \{C^T\}_{l+1,m+1}=\int_{-\infty}^\infty
T_{T-l}(\lambda)T_{T-m}(\lambda)\,d\rho(\lambda), \quad
l,m=0,\ldots,T-1,
\end{equation}
where $d\rho$ is a spectral measure of $H$.

Is is known \cite{SG} that any probability measure with finite
moments on $\mathbb{R}$ give rise to the Jacobi operator, i.e. is
a spectral measure of this operator. In \cite{SG} the authors
posed the question on the characterization of the spectral measure
for the semi-infinite discrete Schr\"odinger operator. The
following theorem answers this question

\begin{theorem}
The measure $d\rho$ is a spectral measure of discrete
Schr\"odinger operator if and only if for every $T$ the matrix
$C^T$ is positive definite and $\det{C^T}=1,$ where
\begin{equation}
\label{C_ij} C^T_{i,j}=\int_{-\infty}^\infty
T_{T-i+1}(\lambda)T_{T-j+1}(\lambda)\,d\rho(\lambda),\quad
i,j=1,\ldots,T.
\end{equation}
\end{theorem}
\begin{proof}
We consider the system (\ref{Jacobi_dyn}) with $a_k=1$. Let
$d\rho$ be a spectral measure of $H$. For every $T\in \mathbb{N}$
we construct the connecting operator $C^T$ (see (\ref{C_T_def}))
using the representation (\ref{C_ij}). According to Theorem
\ref{Teor_Sch}, such $C^T$ is positive definite and $\det{C^T}=1$.

On the other hand, if given measure $d\rho$ satisfies conditions
of the theorem, for every $T$ we can construct $C^T$ by
(\ref{C_ij}) and by Theorem \ref{Teor_Sch} recover coefficients
$b_n$ by (\ref{BK}).
\end{proof}

\noindent{\bf Acknowledgments}

The research of Victor Mikhaylov was supported in part by NIR
SPbGU 11.38.263.2014. A. S. Mikhaylov and V. S. Mikhaylov were
partly supported by VW Foundation program "Modeling, Analysis, and
Approximation Theory toward application in tomography and inverse
problems."


\begin{thebibliography}{99}


\bibitem{AM}
{Avdonin, Sergei; Mikhaylov, Victor} \textit{The boundary control
approach to inverse spectral theory,} Inverse Problems {\bf 26},
2010, no. 4, 045009, 19 pp.

\bibitem{A}
{F. V. Atkinson} \textit{Discrete and continuous boundary
problems}, Acad. Press, 1964.


\bibitem{B07}
{M.~Belishev}, \textit{Recent progress in the boundary control
method}, Inverse Problems, {\bf 23}, (2007).

\bibitem{BM_1}
{M.I.Belishev and V.S.Mikhailov}. \textit{Unified approach to
classical equations of inverse problem theory.} {Journal of
Inverse and Ill-Posed Problems}, 20 (2012), no 4, 461--488.

\bibitem{SG}
{F. Gesztesy, B. Simon}, \textit{$m-$functions and inverse
spectral analisys for dinite and semi-infinite Jacobi matrices},
J. d'Analyse Math. 73 (1997), 267-297

\bibitem{MM}
{A. S. Mikhaylov, V. S Mikhaylov}, \textit{Dynamical inverse
problem for the discrete Schr\"odinger operator.},

\bibitem{S}
{B. Simon}, \textit{The classical moment problem as a self-adjoint
finite difference operator.},  Advances in Math., 137, 1998,
82-203.



\end{thebibliography}
\end{document}